\pdfoutput=1
\documentclass[a4paper]{scrartcl}

\usepackage[utf8]{inputenc}
\usepackage[T1]{fontenc}

\usepackage{amsmath, amsthm, amssymb}
\usepackage{mathtools}
\usepackage{refcount}
\usepackage{mathrsfs}
\usepackage{bbm} 
\usepackage{stmaryrd} 
\usepackage{array,booktabs,multirow}
\usepackage{systeme} 
\usepackage{bigdelim} 
\usepackage{makecell} 
\usepackage{longtable} 
\usepackage{url}
\usepackage{subcaption} 
\DeclareCaptionSubType*[arabic]{figure}

\makeatletter
\let\c@table\c@figure
\let\ftype@table\ftype@figure
\makeatother
\usepackage{adjustbox}
\usepackage{needspace} 
\usepackage{alltt}

\usepackage{enumitem}
\newlist{hypothenum}{enumerate}{3}
\setlist[hypothenum,1]{label=(\roman*)}

\usepackage[table]{xcolor}

\usepackage[referable]{threeparttablex}
\renewlist{tablenotes}{enumerate}{1}
\makeatletter
\setlist[tablenotes]{label=\tnote{\alph*},ref=\alph*,itemsep=\z@,topsep=\z@skip,partopsep=\z@skip,parsep=\z@,itemindent=\z@,labelindent=\tabcolsep,labelsep=.2em,leftmargin=*,align=left,before={\footnotesize}}
\makeatother
\usepackage{genyoungtabtikz}

\theoremstyle{plain}

\newtheorem{theorem}{Theorem}[section]
\newtheorem{proposition}[theorem]{Proposition}

\newtheorem{lemma}[theorem]{Lemma}

\newtheorem{claim}[theorem]{Claim}

\newtheorem{problem}[theorem]{Problem}

\theoremstyle{definition}

\theoremstyle{remark}
\newtheorem{remark}[theorem]{Remark}

\numberwithin{equation}{section}

\newcommand{\eps}{\varepsilon}
\newcommand{\setsuch}[2]{\left\{ #1 \; \middle| \; #2 \right\}}
\newcommand{\restr}[2]{{\left. #1 \right|}_{#2}}

\newcommand{\ZZ}{\mathbb{Z}}
\newcommand{\RR}{\mathbb{R}}
\newcommand{\CC}{\mathbb{C}}
\newcommand{\HH}{\mathbb{H}}

\newcommand{\lie}{\mathfrak}
\newcommand{\twice}{{\scriptstyle 2 \cdot}}

\DeclareMathOperator{\Id}{Id}

\DeclareMathOperator{\SO}{SO}

\usepackage[all,cmtip]{xy}
\newcommand{\fundef}[5]{
\entrymodifiers={+!!<0pt,\fontdimen22\textfont2>}
\xymatrix@R=3pt{\llap{$#1$\;\;} {#2} \ar@{->}[r] & {#3} \\ {#4} \ar@{|->}[r] & {#5}}
} 

\usepackage{xparse}
\makeatletter
\newenvironment{smallermatrix}[1][c]
{\null\,\vcenter\bgroup
  \Let@\restore@math@cr\default@tag
  \baselineskip0pt \lineskip0.4pt \lineskiplimit0pt
  \ialign\bgroup\if#1l\else\hfil\fi$\m@th\scriptstyle##$\if#1r\else\hfil\fi&&\thickspace\hfil
  $\m@th\scriptstyle##$\hfil\crcr
}{%
  \crcr\egroup\egroup\,%
}
\makeatother 

\newcommand\poption[1]{\def\temp{#1}\ifx\temp\empty {} \else \text{\fbox{$\temp$}} \fi}
\newcommand{\threeindwithspace}[3]{
  \begin{smallermatrix}[c]
  \mathstrut\IfValueT{#1}{#1} \\
  \IfValueT{#2}{\poption{#2}}{} \\
  \mathstrut\IfValueT{#3}{#3}
  \end{smallermatrix}%
}
\newcommand{\threeind}[3]{
  \def\haut{#1}
  \def\bas{#3}
  \def\matrix{\threeindwithspace{#1}{#2}{#3}}
  \!
  \ifx\haut\empty
    \ifx\bas\empty \smash{\matrix} \else \smash[t]{\matrix} \fi
  \else
    \ifx\bas\empty \smash[b]{\matrix} \else \matrix \fi
  \fi
  \!%
}

\newtagform{argumenttag}{(}{.\tagargument)}

  \usepackage{zref-abspage}[2010/03/26]
  \makeatletter
  \newcounter{topic@label}
  \renewcommand*{\thetopic@label}{topic@\the\value{topic@label}}
  \global\let\topic@previous\relax
  \global\let\lasttopic\relax
  \newcommand*{\topic}[1]{%
    \begingroup
      \def\topic@put{\topicformat{#1}}%
      \edef\topic@previouslabel{\thetopic@label}%
      \stepcounter{topic@label}%
      \zref@labelbyprops{\thetopic@label}{abspage}%
      \def\topic@current{#1}%
      \ifx\topic@current\topic@previous
        \zifrefundefined{\topic@previouslabel}{%
          \topic@put
        }{%
          \zifrefundefined{\thetopic@label}{%
            \topic@put
          }{%
            \ifnum\zref@extractdefault{\topic@previouslabel}{abspage}{0}=%
                  \zref@extractdefault{\thetopic@label}{abspage}\relax
            \else
              {\topic@put {\tiny { (cont.)}}}
            \fi
          }%
        }%
      \else
        \topic@put
      \fi
      \global\let\topic@previous\topic@current
    \endgroup
    \gdef\lasttopic{\topic{#1}}%
  }
  \newcommand*{\topicformat}[1]{#1}
  \makeatother

\newcommand{\ie}{i.e.\ }

\newcommand\diagramfontsize\footnotesize

\begin{document}

\title{Action of $w_0$ on $V^L$: the special case of $\lie{so}(1,n)$}
\author{Ilia Smilga\thanks{The author was supported by the European Research Council (ERC) under the European Union Horizon 2020 research and innovation programme (ERC starting grant DiGGeS, grant agreement No.\ 715982) and by the Simons Investigator Award 409735 from the Simons Foundation.}}

\maketitle

\begin{abstract}
In this note, we present an algorithm that allows to answer any individual instance of the following question. Let $G_\RR$ be a semisimple real Lie group, and $V$ an irreducible representation of~$G_\RR$. How does the longest element~$w_0$ of the restricted Weyl group~$W$ act on the subspace~$V^L$ of~$V$ formed by vectors that are invariant by~$L$, the centralizer of a maximal split torus of~$G_\RR$? This algorithm comprises two parts. First we describe a complete answer to this question in the particular case where $G_\RR = \SO(1,n)$ for any $n \geq 2$. Then, for an arbitrary $G_\RR$, we show that it suffices to do the computation in a well-chosen subgroup $S_\RR \subset G_\RR$ which is (up to isogeny) the product of several groups that are either compact, abelian or isomorphic to $\SO(1,n)$ for some $n \geq 2$.
\end{abstract}

\section{Introduction}

\subsection{Basic notations}
\label{sec:notations}

We start by setting up the objects involved in the statement or our problem.

\begin{enumerate}
\item Let $\lie{g}$ be a semisimple complex Lie algebra, $\lie{g}_\RR$ some real form of~$\lie{g}$ (so that $\lie{g} = (\lie{g}_\RR)^\CC$).
\item We choose in~$\lie{g}_\RR$ a Cartan subspace~$\lie{a}_\RR$ (an abelian subalgebra of $\lie{g}_\RR$ whose elements are diagonalizable over~$\RR$ and which is maximal for these properties); we set $\lie{a} := (\lie{a}_\RR)^\CC$.
\item We choose in~$\lie{g}$ a Cartan subalgebra~$\lie{h}$ (an abelian subalgebra of $\lie{g}$ whose elements are diagonalizable and which is maximal for these properties) that contains~$\lie{a}$.
\item We denote by~$\lie{l}(\lie{g}_\RR)$, or simply $\lie{l}$ when clear from context, the centralizer of~$\lie{a}$ in~$\lie{g}$. 
\item Let $\Delta$~be the set of roots of~$\lie{g}$ in~$\lie{h}^*$. We shall identify $\lie{h}^*$ with~$\lie{h}$ via the Killing form. We call $\lie{h}_{(\RR)}$ the $\RR$-linear span of~$\Delta$; it is given by the formula $\lie{h}_{(\RR)} = \lie{a}_\RR \oplus i \lie{t}_\RR$, where $\lie{t}_\RR$ is the orthogonal complement of $\lie{a}_\RR$ in~$\lie{h} \cap \lie{g}_\RR$.
\item \label{itm:lex_ord_choice} We choose on~$\lie{h}_{(\RR)}$ a lexicographical ordering that ``puts $\lie{a}_\RR$ first'', \ie such that every vector whose orthogonal projection onto~$\lie{a}_\RR$ is positive is itself positive. We call $\Delta^+$ the set of roots in~$\Delta$ that are positive with respect to this ordering, and we let $\Pi = \{\alpha_1, \ldots, \alpha_r\}$ be the set of simple roots in~$\Delta^+$. Let~$\varpi_1, \ldots, \varpi_r$ be the corresponding fundamental weights.
  \item We call $P$ (resp. $Q$) the weight lattice (resp. root lattice), \ie the abelian subgroup of~$\lie{h}^*$ generated by $\varpi_1, \ldots, \varpi_r$ (resp. by~$\Delta$). Elements of~$P$ are called \emph{integral weights}.
\item We introduce the dominant Weyl chamber:
\[\lie{h}^+ := \setsuch{X \in \lie{h}_{(\RR)}}{\forall \alpha \in \Pi,\quad \alpha(X) \geq 0}.\]
\item In the sequel, all representations are supposed to be finite-dimensional and complex. Recall (\cite[Thm.~5.5]{Kna96first} or \cite[Thms. 9.4 and~9.5]{Hall15}) that to every irreducible representation of~$\lie{g}$, we may associate, in a bijective way, a vector $\lambda \in P \cap \lie{h}^+$ called its \emph{highest weight}. We denote by $\rho_\lambda(\lie{g})$ the irreducible representation of~$\lie{g}$ with highest weight~$\lambda$, and by $V_\lambda(\lie{g})$ the space on which it acts. When clear from context, we will shorten $V_\lambda(\lie{g})$ to~$V_\lambda$.
\item Given a representation $V$ of~$\lie{g}$, we denote by $V^\lie{l} := \setsuch{v \in V}{\forall l \in \lie{l},\;\; l \cdot v = 0}$ the $\lie{l}$-invariant subspace of~$V$. 
\item Choose any connected Lie group $G_\RR$ with Lie algebra $\lie{g}_\RR$. We introduce the \emph{restricted Weyl group} $W := N_{G_\RR}(\lie{a}_\RR)/Z_{G_\RR}(\lie{a}_\RR)$ of~$G_\RR$. Then the action of $W$ on $\lie{a}_\RR$ has as fundamental domain the dominant restricted Weyl chamber $\lie{a}^+ := \lie{h}^+ \cap \lie{a}_\RR$. We define the \emph{longest element} of the restricted Weyl group as the unique element $w_0 \in W$ such that $w_0(\lie{a}^+) = -\lie{a}^+$.
\end{enumerate}

We also introduce some notational conventions that we will follow when talking about specific Lie algebras.

\begin{enumerate}[resume]
  \item We denote by $\lie{sp}_{\twice n} (\CC)$, $\lie{sp}_{\twice n} (\RR)$ and $\lie{sp}_\twice (p,q)$ some Lie algebras that have rank $n$ (or~$p+q$) and a standard representation of dimension $2n$ (or~$2p+2q$). Some authors, such as Bourbaki~\cite{BouGAL456}, denote them respectively by $\lie{sp}_{2n} (\CC)$, $\lie{sp}_{2n} (\RR)$ and $\lie{sp}(2p,2q)$; while other authors, such as Knapp~\cite{Kna96first}, denote them respectively by $\lie{sp}_n (\CC)$, $\lie{sp}_n (\RR)$ and $\lie{sp}(p,q)$.
  \item When $\lie{g} = \lie{g}_\RR^\CC$ is simple, we use the basis $(\eps_1, \ldots, \eps_n)$ introduced in the appendix to \cite{BouGAL456} of a vector space containing $\lie{h}^*_{(\RR)}$. Throughout the paper, we use the Bourbaki conventions \cite{BouGAL456} for the numbering of simple roots and their expressions in the coordinates $\eps_i$.

For the non-simple group $\lie{g} = \lie{so}_4(\CC)$, we furthermore introduce, for consistency with the generic $\lie{so}_{2n}(\CC)$ case (see e.g.\ \cite[\S~II.1, Exm.~4]{Kna96first}), the basis $(\eps_1, \eps_2)$ of the dual Cartan subalgebra $\lie{h}^*_{(\RR)}$ in which the simple roots are $\alpha_1 := \eps_1 + \eps_2$ and $\alpha_2 := \eps_1 - \eps_2$.
  \item Given an integral weight $\lambda \in P$, we always denote $\lambda_1, \ldots, \lambda_n$ its coordinates in the basis we just introduced: 
\begin{equation}
\label{eq:coordinate_definition}
\lambda =: \sum_{i=1}^n \lambda_i \eps_i.
\end{equation}
\end{enumerate}

\subsection{Statement of problem}
\label{sec:statement}

This work is part of a larger effort to characterize, for a given semisimple real Lie algebra~$\lie{g}_\RR$, the representations $V$ of~$\lie{g}_\RR$ for which the action of~$w_0$ on $V^\lie{l}$ is nontrivial. This action is always well-defined, and depends only on the algebra $\lie{g}_\RR$, not on the group $G_\RR$: see \cite[\S~2.1]{Smi20}. This problem has a geometric motivation, related to the study of groups of affine transformations acting properly: see the introduction to \cite{Smi20barx}\footnote{I refer to the arXiv version of that paper, as it contains some extra details to which I will later specifically point. In the published version \cite{Smi20b}, some of the cited propositions are either missing altogether, or present but numbered differently.}
 for more details. It naturally splits into two subproblems (see \cite{Smi20} for a more extended discussion):
\begin{problem}
\label{nontrivial_Vl}
Given a semisimple Lie algebra~$\lie{g}$ and a dominant integral weight~$\lambda$, give a simple necessary and sufficient condition for having $V_\lambda^\lie{l} \neq 0$.
\end{problem}
\begin{problem}
\label{nontrivial_w0_action_simple}
Given a simple Lie algebra~$\lie{g}$ and a dominant integral weight~$\lambda$, assuming that $V_\lambda^\lie{l} \neq 0$, give:
\begin{hypothenum}
\item a simple necessary and sufficient condition for having $\restr{w_0}{V_\lambda^\lie{l}} = \pm \Id$;
\item a criterion to determine the actual sign.
\end{hypothenum}
\end{problem}

In~\cite{Smi20barx}, we have already completely solved Problem~\ref{nontrivial_Vl}. In~\cite{LFlSm}, we have solved Problem~\ref{nontrivial_w0_action_simple} in the case where $\lie{g}$ is split. In this note, we present an algorithm that allows to solve Problem~\ref{nontrivial_w0_action_simple} on a case-by-case basis, i.e.\ to compute $\restr{w_0}{V_\lambda^\lie{l}}$ for any \emph{single} value of $\lie{g}_\RR$ and $\lambda$. This algorithm was announced in \cite{Smi22}, in which it was used to compute some numerical data and to formulate a partial conjectural answer to Problem~\ref{nontrivial_w0_action_simple}.

\subsection{Summary of results}
\label{sec:summary}

This paper contains two theorems:
\begin{itemize}
\item Theorem~\ref{algo_for_so_n_1} completely solves Problem~\ref{nontrivial_w0_action_simple} for a key subset of the algebras $\lie{g}_\RR$, namely for $\lie{g}_\RR = \lie{so}(1, n)$ for any $n \geq 2$.
\item Theorem~\ref{reduction_to_so_n_1} then essentially explains how the computation of $\restr{w_0}{V_\lambda^\lie{l}}$ in an arbitrary Lie algebra $\lie{g}_\RR$ can be reduced to the case of $\lie{so}(1, n)$.
\end{itemize}
Taken together, these two theorems thus provide the desired algorithm.

Moreover, I have actually implemented this algorithm in the LiE software package \cite{LiE}. The code, and its documentation, can be found in the ancillary files.

\section{Solution in the case of $\lie{so}(1,n)$}

\begin{theorem}
\label{algo_for_so_n_1}
Assume that $\lie{g}_\RR = \lie{so}(1,n)$, for some $n \geq 2$. Let $\lambda = \sum_{i=1}^r \lambda_i \eps_i \in \lie{h}^*$ be a dominant weight of $\lie{g}$. Let \eqref{eq:star} denote the condition
\begin{equation}
\label{eq:star} \tag{*}
\begin{cases}
\text{if } n \geq 3:\quad \lambda = \lambda_1 \eps_1 + \lambda_2 \eps_2 \text{ for some } \lambda_1, \lambda_2 \in \ZZ \text{ with } \lambda_1 + \lambda_2 \equiv 0 \pmod{2}; \\
\text{if } n = 2:\quad \lambda = \lambda_1 \eps_1 \text{ for some } \lambda_1 \in \ZZ.
\end{cases}
\end{equation}
%
%
%
Then:
\begin{hypothenum}
\item \label{itm:V_L_nontrivial} \cite{Smi20barx} We have $V^{\lie{l}(\lie{g}_\RR)}_\lambda \neq 0$ if and only if $\lambda$ satisfies \eqref{eq:star}.
\item \label{itm:V_L_dimension} If $\lambda$ satisfies \eqref{eq:star}, then in fact $\dim V^{\lie{l}(\lie{g}_\RR)}_\lambda = 1$.
\item \label{itm:action_of_w0} If $\lambda$ satisfies \eqref{eq:star}, then
\begin{equation}
\label{eq:action_of_w0}
\restr{w_0}{V^{\lie{l}(\lie{g}_\RR)}_\lambda} = (-1)^{\lambda_1} \Id.
\end{equation}
\end{hypothenum}
\end{theorem}

\begin{proof}~
\begin{hypothenum}
\item follows from the Main Theorem in \cite{Smi20barx}. More precisely:
\begin{itemize}
\item For $n = 2$ or $n \geq 4$, it follows from point (i) of that theorem, and can be read off the $B_r$ or $D_r$ row of \cite[Table~1]{Smi20barx}. Note that in the $B_r$ case, the inequality $\lambda_{2r-2p+1} > 0$ is (by convention) tautologically false for $r \geq 1$. For $r = 1$ (corresponding to $n = 2$), on the contrary, it is almost tautologically true, which is why the parity condition $\lambda_1 + \lambda_2 \equiv 0 \pmod{2}$ does not appear then.
\item For $n = 3$, it follows from point (iv) of that theorem. Indeed, recall that $\lie{so}(3,1) \simeq \lie{sl}_2(\CC)$ is ``already'' a complex Lie algebra; and note that in the basis $(\eps_1, \eps_2)$ we have chosen, the formula given in \cite[Table~2]{Smi20barx} for the root lattice $Q$ still applies, even for the non-simple algebra $\lie{so}(3,1)^\CC \simeq \lie{so}_4(\CC)$.
\end{itemize}
\item We distinguish three cases, depending on the value of $n$:
\begin{itemize}
\item The case $n = 2$ is straightforward. (As a warm-up exercise, the reader can also treat this case by the technique outlined below for the $n \geq 4$ case. A crucial point is that the pair of columns \begin{tikzpicture}[baseline={([yshift=-.5ex]current bounding box.center)}]
\Yfillcolor{black!20}
\tgyoung(0cm,0cm,:|1)
\Yfillopacity{0}
\Ylinecolor{black!30}
\tgyoung(0cm,0cm,1<\overline{1}>)
\Ylinecolor{black}
\tgyoung(0cm,0cm,_2)
\end{tikzpicture} is admissible in type $B_1$, but not admissible in types $B_r$ with $r \geq 2$ or $D_r$.)
\item The case $n = 3$ is presumably well-known. Recall that, since $\lie{so}_4(\CC) \simeq \lie{so}_3(\CC) \oplus \lie{so}_3(\CC)$, every irreducible representation of $V_{c_1 \varpi_1 + c_2 \varpi_2}$ of $\lie{so}_4(\CC)$ is isomorphic to the tensor product of the two representations $V_{c_1 \varpi_1}$ and $V_{c_2 \varpi_2}$ of the respective simple summands. Since both of these summands are isomorphic to $\lie{so}_3(\CC)$, all the weights of $V_{c_1 \varpi_1}$ and of $V_{c_1 \varpi_1}$ have multiplicity~$1$; hence the same holds for the weights of the product representation $V_{c_1 \varpi_1 + c_2 \varpi_2}$.
\item Finally let $n \geq 4$, and assume from now on that $\lambda$ satisfies \eqref{eq:star}. Then the machinery developed in \cite[Sec.~5]{Smi20barx} (an explicit version of Littelmann's path model \cite{Lit95}) tells
us\footnote{Combining \cite[Cor.~5.5]{Smi20barx} with \cite[Prop.~5.30]{Smi20barx} immediately yields this claim, even though only a weaker form is stated there: \cite[Cor.~5.31]{Smi20barx} only discerns zero from nonzero.}
that the dimension of $V^{\lie{l}(\lie{g}_\RR)}_\lambda$ is equal to the number of doubled Young tableaux (see \cite[Def.~5.23]{Smi20barx}) satisfying the conditions (H1) through (H7) of \cite[Cor.~5.31]{Smi20barx}. Let $\mathcal{T}$ be such a doubled Young tableau. We then successively derive several properties of it, which, taken together, will constrain it to a unique value.
\begin{itemize}
\item $\mathcal{T}$ may contain no symbols other than $1$, $2$, $\overline{2}$ or $\overline{1}$. This follows from \cite[Prop.~6.1.(i)]{Smi20barx}, applied to $x=1$ (keep in mind that $\Theta(\lie{g}_{\RR}) = \Pi_{[2,r]}$ in our case: see \cite[Table~8]{Smi20barx}). Indeed, it yields the inequality $t \leq \frac{h}{2} + 1$, where $h$ refers to the height of the Young tableau~$\mathcal{T}$ (which in our case is at most~$2$ since $\lambda_3 = 0$), and $t$ refers to the largest absolute value of a symbol that may appear in $\mathcal{T}$.
\item We know that $\mathcal{T}$ has nondecreasing values along each row (for the obvious order \eqref{eq:obvious_order}: use \cite[Rem.~5.17]{Smi20barx}, or, more simply, see the discussion preceding \eqref{eq:obvious_order} below). Consequently, it can be entirely described by the values of the parameters $0 \leq a_1 \leq a_2 \leq a_{\overline{2}} \leq a_{\overline{1}} = \lambda_1$ and $0 \leq b_1 \leq b_2 \leq b_{\overline{2}} \leq b_{\overline{1}} = \lambda_2$ (see Figure~\ref{fig:explicit_tableau}), where we set, for $s \in \{1, 2, \overline{2}, \overline{1}\}$,
\[a_s := \frac{1}{2} \# \threeind{}{\preceq s}{1} \mathcal{T}
\quad\text{and}\quad
b_s := \frac{1}{2} \# \threeind{}{\preceq s}{2} \mathcal{T}\]
in the notations of \cite[Def.~4.2]{Smi20barx}. In other terms, $a_s$ (resp. $b_s$) is half the number of boxes in the first (resp. second) row of $\mathcal{T}$ containing symbols not exceeding $s$.
\item $b_1 = 0$, since $\mathcal{T}$ also has standard columns, i.e.\ strictly increasing values along each column.
\item $b_2 = b_1 = 0$ and $a_{\overline{2}} = a_2$. Indeed, since $\mathcal{T}$ is $\alpha_2$-codominant, the leftmost (resp. rightmost) occurrence in $\mathcal{T}$ of any of the symbols in $\{2, \overline{2}\}$ must be a $\overline{2}$ (resp. a $2$). (This can be seen as a special case of \cite[Lem.~6.4]{Smi20barx} applied to $s = 2$, but here the proof is simpler because the symbols $3$ and $\overline{3}$ do not appear at all.) And moreover, that leftmost (resp. rightmost) occurrence must necessarily be in the second (resp. first) row (this can be read off \cite[(6.9)]{Smi20barx}, or, once again, checked directly by hand).
\item $a_2 = a_1 + b_{\overline{2}} = \frac{\lambda_1 + \lambda_2}{2}$. This follows from the fact that $\mathcal{T}$ is null, i.e.\ that, for every $s$, it contains the same number of symbols $s$ and $\overline{s}$.
\item $b_{\overline{2}} \leq a_1$: otherwise we would have a
\begin{tikzpicture}[baseline={([yshift=-.5ex]current bounding box.center)}]
\Yfillcolor{black!20}
\tgyoung(0cm,0cm,,|1)
\Yfillopacity{0}
\Ylinecolor{black!30}
\tgyoung(0cm,0cm,2,<\overline{2}>)
\Ylinecolor{black}
\tgyoung(0cm,0cm,|2)
\end{tikzpicture}
column, which is forbidden (by (H1), i.e.\ strong standardness: see \cite[Def.~5.14]{Smi20barx}).
\item If $b_{\overline{2}} < a_1$, then $b_{\overline{2}} = b_{\overline{1}} = \lambda_2$: otherwise we would have a
\begin{tikzpicture}[baseline={([yshift=-.5ex]current bounding box.center)}]
\Yfillcolor{black!20}
\tgyoung(0cm,0cm,,|1)
\Yfillopacity{0}
\Ylinecolor{black!30}
\tgyoung(0cm,0cm,1,<\overline{1}>)
\Ylinecolor{black}
\tgyoung(0cm,0cm,|2)
\end{tikzpicture}
column, which is once again forbidden by strong standardness.
\end{itemize}
All of these constraints taken together then force a unique value for each of the parameters $a_s$ and $b_s$, namely:
\begin{equation}
\label{eq:so_n_1_invariant_tableau}
\begin{cases}
a_1 = \max(\frac{\lambda_1 + \lambda_2}{4}, \frac{\lambda_1 - \lambda_2}{2}), \\
a_2 = a_{\overline{2}} = \frac{\lambda_1 + \lambda_2}{2}, \\
a_{\overline{1}} = \lambda_1,
\end{cases}
\qquad\text{and}\qquad
\begin{cases}
b_1 = b_2 = 0, \\
b_{\overline{2}} = \min(\frac{\lambda_1 + \lambda_2}{4}, \lambda_2), \\
b_{\overline{1}} = \lambda_2.
\end{cases}
\end{equation}
This shows that $\dim V^{\lie{l}(\lie{g}_\RR)}_\lambda \leq 1$. The opposite inequality is already given in \ref{itm:V_L_nontrivial}, so the conclusion follows.

\begin{figure}
\centering
\caption{\label{fig:explicit_tableau} The unique doubled Young tableau $\mathcal{T}$ satisfying conditions (H1) through (H7) of \cite[Cor.~5.31]{Smi20barx}, for any $\lambda$ whose coordinates $\lambda_1$ and $\lambda_2$ are integer with even sum. (We assume $n \geq 4$ here.)}
\begin{subfigure}[t]{0.5\textwidth}
\centering
\caption{\label{fig:lambda2_small} The case $\lambda_2 \leq \frac{\lambda_1}{3}$.\\ (Here $(\lambda_1, \lambda_2) = (4, 2)$).}
\vspace{3mm}
\begin{tikzpicture}[x=0.457cm,y=0.457cm]
\Yfillcolor{black!20}
\tgyoung(0cm,0cm,::::::|1|1,|1|1|1|1)
\Yfillopacity{0}
\Ylinecolor{black!30}
\tgyoung(0cm,0cm,111222<\overline{1}><\overline{1}>,<\overline{2}><\overline{2}><\overline{2}><\overline{1}>)
\Ylinecolor{black}
\tgyoung(0cm,0cm,_2_2_2_2,_2_2)
\draw[->](3,2)--(3,1);
\node[above] at (3,2) {$a_1 = \frac{\lambda_1+\lambda_2}{4}$};
\draw[->](6,4)--(6,1);
\node[above] at (6,4) {$a_2 = a_{\overline{2}} = \frac{\lambda_1+\lambda_2}{2}$};
\draw[->](8,2)--(8,1);
\node[above] at (8,2) {$a_{\overline{1}} = \lambda_1$};
\draw[->](0,-2)--(0,-1);
\node[below] at (0,-2) {$b_1 = b_2 = 0$};
\draw[->](3,-3)--(3,-1);
\node[below] at (3,-3) {$b_{\overline{2}} = \frac{\lambda_1+\lambda_2}{2}$};
\draw[->](5,-2)--(4,-1);
\node[below right] at (5,-2) {$b_{\overline{1}} = \lambda_2$};
\end{tikzpicture}
\end{subfigure}%
~
\begin{subfigure}[t]{0.5\textwidth}
\centering
\caption{\label{fig:lambda2_big} The case $\lambda_2 \geq \frac{\lambda_1}{3}$.\\ (Here $(\lambda_1, \lambda_2) = (5, 1)$).}
\vspace{3mm}
\begin{tikzpicture}[x=0.457cm,y=0.457cm]
\Yfillcolor{black!20}
\tgyoung(0cm,0cm,::::::|1|1|1|1,|1|1)
\Yfillopacity{0}
\Ylinecolor{black!30}
\tgyoung(0cm,0cm,111122<\overline{1}><\overline{1}><\overline{1}><\overline{1}>,<\overline{2}><\overline{2}>)
\Ylinecolor{black}
\tgyoung(0cm,0cm,_2_2_2_2_2,_2)
\draw[->](3,2)--(4,1);
\node[above] at (3,2) {$a_1 = \frac{\lambda_1-\lambda_2}{2}$};
\draw[->](6,4)--(6,1);
\node[above] at (6,4) {$a_2 = a_{\overline{2}} = \frac{\lambda_1+\lambda_2}{2}$};
\draw[->](10,2)--(10,1);
\node[above] at (10,2) {$a_{\overline{1}} = \lambda_1$};
\draw[->](0,-2)--(0,-1);
\node[below] at (0,-2) {$b_1 = b_2 = 0$};
\draw[->](3,-2)--(2,-1);
\node[below right] at (3,-2) {$b_{\overline{2}} = b_{\overline{1}} =\lambda_2$};
\end{tikzpicture}
\end{subfigure}
\end{figure}

Alternatively, one may also directly verify that the values \eqref{eq:so_n_1_invariant_tableau} define a tableau~$\mathcal{T}$ that satisfies the required conditions (shown in Figure~\ref{fig:explicit_tableau}). This is mostly straightforward: most of the properties are true by construction. The only slightly nontrivial property is (H3) (partition into admissible pairs, in the sense of \cite[Rem.~3.4]{Lit90}). Let us check it:
\begin{itemize}
\item Most of the values of the parameters $a_s$ and $b_s$ are integer. This means that, for most even (see \cite[Rem.~5.29]{Smi20barx}) values of $j$, the columns $\threeind{j}{}{} \mathcal{T}$ and $\threeind{j-1}{}{} \mathcal{T}$ simply coincide, so automatically form an admissible pair.
\item The only exception is $a_1$ and $b_{\overline{2}}$, if $\frac{\lambda_1 + \lambda_2}{2}$ is odd and $\lambda_2 > \frac{\lambda_1}{3}$. In this case, we get a single even value of $j$ (namely $j = 2a_1 + 1$) for which the $j-1$-st and $j$-th columns differ. This pair of columns is equal to
\begin{tikzpicture}[baseline={([yshift=-.5ex]current bounding box.center)}]
\Yfillcolor{black!20}
\tgyoung(0cm,0cm,,|1|1:)
\Yfillopacity{0}
\Ylinecolor{black!30}
\tgyoung(0cm,0cm,1:2,<\overline{2}><\overline{1}>)
\Ylinecolor{black}
\tgyoung(0cm,0cm,|2|2)
\end{tikzpicture};
when taken in the reverse order (because of the order inversion between tableaux and paths, see \cite[(5.20)]{Smi20barx}), it clearly is admissible (see \cite[Prop.~5.21]{Smi20barx}).
\end{itemize}
Also note that in type $D_r$, we have $r > 2$, which gives two significant simplifications:
\begin{itemize}
\item For (H2) (nondecreasingness in horizontal direction), in type $D_r$, the tableau $\mathcal{T}$ does not contain any symbols $r$ or $\overline{r}$. This means that we can ignore the subtleties of \cite[Def.~5.16]{Smi20barx}: in our case the order $\preceq^{\lie{g}}_Y$ coincides with the ordinary order $\preceq_Y$; explicitly:
\begin{equation}
\label{eq:obvious_order}
1 \preceq^{\lie{g}}_Y 2 \preceq^{\lie{g}}_Y \overline{2} \preceq^{\lie{g}}_Y \overline{1}.
\end{equation}
\item We can ignore (H6), which is automatically true since the tableau $\mathcal{T}$ only has height $2$ and does not contain any columns of height $r$.
\end{itemize}
\end{itemize}
\item Assume once again that $\lambda$ satisfies \eqref{eq:star}. Recall that the set of such $\lambda$ (denoted in \cite{Smi20barx} by $\mathcal{M}_{\lie{l}\text{-inv}}$) is an additive monoid \cite[Prop.~2.5]{Smi20barx}.

Since $V^{\lie{l}(\lie{g}_\RR)}_\lambda$ is then $1$-dimensional and $w_0$ acts on it as an involution, we necessarily have
\begin{equation}
\restr{w_0}{V^{\lie{l}(\lie{g}_\RR)}_\lambda} = \sigma(\lambda) \Id
\end{equation}
for some function $\sigma: \mathcal{M}_{\lie{l}\text{-inv}}(\lie{g}_\RR) \to \{\pm 1\}$ to be determined. Moreover, from the properties of the Cartan product (compare the proof of \cite[Prop.~1.(iii)]{Smi20}) it immediately follows that $\sigma$ is a semigroup morphism. So it suffices to establish the identity \eqref{eq:action_of_w0} for a basis of the monoid $\mathcal{M}_{\lie{l}\text{-inv}}$.

To compute this basis, we use the well-known formulas (recalled in \cite[Table~2]{Smi20barx} for simple $\lie{g}$, but also valid for $\lie{so}_4(\CC)$) describing the Weyl chamber $\lie{h}^+$:
\begin{itemize}
\item for $n = 2$, the basis contains only $\lambda = \eps_1$, which corresponds to the adjoint representation;
\item for $n = 3$, the basis contains $\lambda = \eps_1 \pm \eps_2$, which correspond to the two irreducible summands of the adjoint representation;
\item for $n \geq 4$ (regardless of the parity), the basis contains $\lambda = \eps_1+\eps_2$, which corresponds to the adjoint representation, and $\lambda = 2\eps_1$, which corresponds to the main irreducible summand of the second symmetric power of the standard representation.
\end{itemize}
Checking \eqref{eq:action_of_w0} in all of these cases is then a straightforward computation. \qedhere
\end{hypothenum}
\end{proof}

\section{Reduction to the case of $\lie{so}(1,n)$}

\begin{theorem}
\label{reduction_to_so_n_1}
The algebra $\lie{g}_\RR$ contains a reductive subalgebra $\lie{s}_\RR$ with the following properties:
\begin{hypothenum}
\item \label{itm:made_of_so_n_1} Every simple summand of $\lie{s}_\RR$ is either abelian, compact, or isomorphic to $\lie{so}(1, n)$ for some $n \geq 2$.
\item \label{itm:same_Cartan} $\lie{s}_\RR$ shares with $\lie{g}_\RR$ the same Cartan subalgebra $\lie{h}$ and Cartan subspace $\lie{a}_\RR$.
\item \label{itm:same_w0} Denote by $S_\RR$ the connected subgroup of $G_\RR$ with Lie algebra $\lie{s}_\RR$. Then the longest element $w_0$ of the restricted Weyl group of $S_\RR$ is compatible with the $w_0$ of $G_\RR$ (i.e.\ some representative of $w_0$ in $S_\RR$ is also a representative of $w_0$ in $G_\RR$).
\end{hypothenum}
\end{theorem}
Moreover, this subalgebra $\lie{s}_\RR$ can be explicitly described: see Appendix. 

We start by recalling the following lemma, slightly generalizing a construction by \cite{AK84}:
\begin{lemma}
\label{strongly_orth_decomp}
Every root system $\Sigma$ (not necessarily reduced) has a subset $\Xi$ of pairwise strongly orthogonal roots such that the longest element $w_0$ of $\Sigma$'s Weyl group is equal to the product $\prod_{\alpha \in \Xi} s_{\alpha}$ of the reflections with respect to these roots.
\end{lemma}
(We recall that two roots $\alpha$ and~$\beta$ are called \emph{strongly} orthogonal if neither $\alpha + \beta$ nor $\alpha - \beta$ belong to the root system.)
\begin{proof}
Clearly it suffices to do this separately within each irreducible component. For irreducible reduced root systems, these subsets are given in \cite[Table~2]{LFlSm}. For $BC_n$, the same set as for $C_n$ (namely $\setsuch{2\eps_i}{1 \leq i \leq n}$) works.
\end{proof}

\begin{proof}[Proof of Theorem~\ref{reduction_to_so_n_1}]
We define $\lie{s}_\RR$ as the sum
\[\lie{s}_\RR := \lie{g}_\RR^0 \oplus \bigoplus_{\alpha \in \Xi} \left( \lie{g}_\RR^{\alpha} \oplus \lie{g}_\RR^{-\alpha} \right)\]
of the restricted root space $\lie{g}_\RR^0 = \lie{l} \cap \lie{g}_\RR$ corresponding to $0$ and those corresponding to the elements of the set $\Xi$ given by Lemma~\ref{strongly_orth_decomp} for the restricted root system $\Sigma$ of $\lie{g}_\RR$ and their negatives. It remains to check that this vector subspace of $\lie{g}_\RR$ has indeed the desired properties.
\begin{itemize}
\item The fact that $\lie{s}_\RR$ is a Lie subalgebra of $\lie{g}_\RR$ immediately follows from the identity $[\lie{g}_\RR^\alpha, \lie{g}_\RR^\beta] \subset \lie{g}_\RR^{\alpha + \beta}$ \cite[Prop.~6.40(b)]{Kna96first}: indeed, the sum of any two elements of $\{0\} \cup \setsuch{\pm \alpha}{\alpha \in \Xi}$ either remains within this set or (by strong orthogonality) falls outside $\Sigma$.
\item Since the set $\{0\} \cup \setsuch{\pm \alpha}{\alpha \in \Xi}$ is centrally symmetric, $\lie{s}_\RR$ is reductive (by \cite[Prop.~6.40(c) and Cor.~6.29]{Kna96first}).
\item Since $\lie{a}_\RR$ (resp. $\lie{h}$) is by construction contained in $\lie{s}_\RR$ (resp. in $\lie{s} := \lie{s}_\RR^\CC$), clearly it is also a Cartan subspace (resp. subalgebra) of $\lie{s}_\RR$ (resp. of $\lie{s}$).
\item To check property~\ref{itm:same_w0} (compatibility of the $w_0$ elements), note that for each $\alpha \in \Xi$, we can explicilty construct (by \cite[Prop.~6.52(c)]{Kna96first}) a representative $\overline{\sigma_\alpha} \in G_\RR$ of the reflection $s_\alpha \in W(G_\RR)$; and by construction this representative lies in~$S_\RR$. Hence their product $\prod_{\alpha \in \Xi} \overline{\sigma_\alpha}$ lies in $S_\RR$, and is a representative of $w_0(G_\RR) \in W(G_\RR)$. Since it normalizes $\lie{a}_\RR$ and acts by $-\Id$ on the subspace spanned by $\Xi$, it is also a representative of $w_0(S_\RR) \in W(S_\RR)$.
\item Finally, for property~\ref{itm:made_of_so_n_1}, note that, by construction, $\lie{s}_\RR$ has a restricted root system of type $A_1^s$ (where $s$ is the cardinality of $\Xi$). It is a fact that the restricted root system of every noncompact simple Lie algebra is irreducible (this can be checked case by case in the classification, or see \cite[Lem~15.5.6]{Spr98} for an abstract argument): hence every simple summand of $\lie{s}_\RR$ (ignoring the abelian part) is either compact or has a restricted root system of type $A_1$. To conclude, we check (by looking at the tables) that the only simple Lie algebras that have a restricted root system of type $A_1$ are the $\lie{so}(1,n)$ for $n \geq 2$. \qedhere
\end{itemize}

We can then find the set of roots of $\lie{s}_\RR$, as the set of preimages of $\{0\} \cup \setsuch{\pm \alpha}{\alpha \in \Xi}$ under the projection map from $\lie{h}$ to $\lie{a}$.
\end{proof}

\appendix
\section*{Appendix: making things explicit}
\renewcommand{\thesection}{A}

We can actually compute the algebra $\lie{s}_\RR$ for every real form $\lie{g}_\RR$. Obviously it can be done component by component (see also \cite[Thm.~2.9]{Smi20} for some more trivialities about reduction to simple components). So we will henceforth assume that $\lie{g}_\RR$ is simple.

We recall that a simple real Lie algebra $\lie{g}_\RR$ can be of two types:
\begin{itemize}
\item Either $\lie{g}_\RR$ is actually complex, i.e.\ is obtained from some simple complex Lie algebra $\lie{g}_\CC$ by restriction of scalars, so that its complexification $\lie{g}_\RR^\CC$ is isomorphic to $\lie{g}_\CC \oplus \lie{g}_\CC$.
\item Or $\lie{g}_\RR$ is absolutely simple, i.e.\ the complexification $\lie{g}_\RR^\CC$ is simple.
\end{itemize}

\begin{table}[p]
  \caption[caption]{\label{tab:s_roots} The Lie algebra $\lie{s}_\RR$ from Theorem~\ref{reduction_to_so_n_1} for all absolutely simple real Lie algebras~$\lie{g}_\RR$. We first specify the isomorphism type of~$\lie{s}_\RR$, and then, for each direct summand of~$\lie{s}_\RR$, we list its simple roots. We \colorbox{yellow}{highlight} all the roots that correspond to unblackened nodes in the Satake diagram of $\lie{s}_\RR$ (i.e.\ the set $\Pi_1$ in the notations of \cite[5.4.3°]{OV90}).
\\\hspace{\textwidth} 
To fully reconstruct that Satake diagram, it should be noted that the roots of each simple summand are always listed in the Bourbaki order; and that arrows never occur, except as noted in the footnotes.}
  \addtocounter{table}{-1}
  \centering\bigskip
  \makebox[\textwidth][c]{
  \begin{threeparttable}
  \begin{tabular}[t]{llll}
    $\lie{g}$ & $\lie{g}_\RR$ & Isom. type of $\lie{s}_\RR$ & Roots of each summand of $\lie{s}$ \\
    \midrule
    \multirow[t]{11}{*}{$\underset{r \geq 1}{A_r}$}
    & \multirow{2}{*}{\makecell[l]{$\lie{sl}_{r+1}(\RR)$}}
      & $\lie{so}(1,2)^{\lfloor \frac{r+1}{2} \rfloor}$ & $\left(\colorbox{yellow}{$\eps_i - \eps_{r+2-i}$}\right)_{i = 1, \ldots, \lfloor \frac{r+1}{2} \rfloor}$ \\
    & & $\oplus\; \RR^{\lfloor \frac{r}{2} \rfloor}$ \\ \cmidrule{2-4}
    & \multirow{3}{*}{\makecell[l]{$\lie{su}(p,r+1-p)$ \\ \quad \footnotesize $0 \leq p < \frac{r+1}{2}$}}
      & $\lie{so}(1,2)^p$ & $\left(\colorbox{yellow}{$\eps_i - \eps_{r+2-i}$}\right)_{i = 1, \ldots, p}$ \\
    & & $\oplus\; \lie{su}(r+1-2p)$ & $\alpha_{p+1}, \ldots, \alpha_{r-p}$ \\
    & & $\oplus\; \RR^p$ & \\ \cmidrule{2-4}
    & \multirow{2}{*}{\makecell[l]{$\lie{su}(p,p)$\\ \quad \footnotesize $2p - 1 = r$}}
      & $\lie{so}(1,2)^p$ & $\left(\colorbox{yellow}{$\eps_i - \eps_{r+2-i}$}\right)_{i = 1, \ldots, p}$ \\
    & & $\oplus\; \RR^{p-1}$ & \\ \cmidrule{2-4}
    & \multirow{2}{*}{\makecell[l]{$\lie{sl}_m(\HH)$\\ \quad \footnotesize $2m - 1 = r$, $m$ even}}
      & $\lie{so}(1,5)^{\frac{m}{2}}$ & $\left(\colorbox{yellow}{$\eps_{2i} - \eps_{r+2-2i}$},\; \alpha_{2i-1},\; \alpha_{r+2-2i}\right)_{i = 1, \ldots, \frac{m}{2}}$ \\
    & & $\oplus\; \RR^{\frac{m}{2}-1}$ & \\ \cmidrule{2-4}
    & \multirow{3}{*}{\makecell[l]{$\lie{sl}_m(\HH)$\\ \quad \footnotesize $2m - 1 = r$, $m$ odd}}
      & $\lie{so}(1,5)^{\frac{m-1}{2}}$ & $\left(\colorbox{yellow}{$\eps_{2i} - \eps_{r+2-2i}$},\; \alpha_{2i-1},\; \alpha_{r+2-2i}\right)_{i = 1, \ldots, \frac{m-1}{2}}$ \\
    & & $\oplus\; \lie{so}_3(\RR)$ & $\alpha_m$ \\
    & & $\oplus\; \RR^{\frac{m-1}{2}}$ & \\
    \midrule
    \multirow[t]{4}{*}{$\underset{n \geq 3,\; n \neq 4}{B_{\frac{n-1}{2}} \text{ or } D_{\frac{n}{2}}}$}
    & \multirow{2}{*}{\makecell[l]{$\lie{so}(p,n-p)$\\ \quad \footnotesize $0 \leq p \leq \frac{n}{2}$, $p$ even}}
      & $\lie{so}(1,2)^p$ & $(\colorbox{yellow}{$\eps_{2i-1} \pm \eps_{2i}$})_{i = 1, \ldots, \frac{p}{2}}$ \\
    & & $\oplus\; \lie{so}_{n-2p}(\RR)$ & $\alpha_{p+1}, \ldots, \alpha_{\lfloor \frac{n}{2} \rfloor}$\tnotex{tnote:so2} \\ \cmidrule{2-4}
    & \multirow{2}{*}{\makecell[l]{$\lie{so}(p,n-p)$\\ \quad \footnotesize $0 \leq p \leq \frac{n}{2}$, $p$ odd}}
      & $\lie{so}(1,2)^{p-1}$ & $(\colorbox{yellow}{$\eps_{2i-1} \pm \eps_{2i}$})_{i = 1, \ldots, \frac{p-1}{2}}$ \\
    & & $\oplus\; \lie{so}(1, n+1-2p)$ & $\colorbox{yellow}{$\alpha_p$}, \alpha_{p+1}, \ldots, \alpha_{\lfloor \frac{n}{2} \rfloor}$\tnotex{tnote:so2} \tnote{,} \tnotex{tnote:so13} \\
    \midrule
    \multirow[t]{3}{*}{$\underset{r \geq 1}{C_r}$}
    & $\lie{sp}_{\twice r}(\RR)$
      & $\lie{so}(1,2)^r$ & $(\colorbox{yellow}{$2\eps_i$})_{i = 1, \ldots, r}$ \\ \cmidrule{2-4}
    & \multirow{2}{*}{\makecell[l]{$\lie{sp}_\twice(p,r-p)$\\ \quad \footnotesize $0 \leq p \leq \frac{r}{2}$}}
      & $\lie{so}(1,4)^p$ & $(\colorbox{yellow}{$2\eps_{2i}$}, \eps_{2i-1} - \eps_{2i})_{i = 1, \ldots, p}$ \\
    & & $\oplus\; \lie{sp}_\twice(r-2p)$ & $\alpha_{2p+1}, \ldots, \alpha_r$ \\
    \midrule
    \multirow[t]{5}{*}{$\underset{r \geq 3}{D_r}$}
    & \multirow{2}{*}{\makecell[l]{$\lie{so}^*(2r)$\\ \quad \footnotesize $r$ even}}
      & $\lie{so}(1,2)^{\frac{r}{2}}$ & $(\colorbox{yellow}{$\eps_{2i-1} - \eps_{2i}$})_{i = 1, \ldots, \frac{r}{2}}$ \\
    & & $\oplus\; \lie{so}_3(\RR)^{\frac{r}{2}}$ & $(\eps_{2i-1} + \eps_{2i})_{i = 1, \ldots, \frac{r}{2}}$ \\ \cmidrule{2-4}
    & \multirow{3}{*}{\makecell[l]{$\lie{so}^*(2r)$\\ \quad \footnotesize $r$ odd}}
      & $\lie{so}(1,2)^{\frac{r-1}{2}}$ & $(\colorbox{yellow}{$\eps_{2i-1} - \eps_{2i}$})_{i = 1, \ldots, \frac{r-1}{2}}$ \\
    & & $\oplus\; \lie{so}_3(\RR)^{\frac{r-1}{2}}$ & $(\eps_{2i-1} + \eps_{2i})_{i = 1, \ldots, \frac{r-1}{2}}$ \\
    & & $\oplus\; \RR$ & \\
    \bottomrule
  \end{tabular}
  \vspace{2mm}
  \begin{tablenotes}
    \item\label{tnote:so2} If $n - 4 \lfloor \frac{p}{2} \rfloor = 2$, this list is empty instead (as $\lie{so}_2(\RR) \simeq \lie{so}(1,1) \simeq \RR$).
    \item\label{tnote:so13} If $n + 1 - 2p = 3$, then \emph{both} of the roots $\alpha_{\frac{n}{2}-1}$ and $\alpha_{\frac{n}{2}}$ are noncompact, and are joined together by an arrow.
  \end{tablenotes}
  \end{threeparttable}
  }
\end{table}

\clearpage

\begin{small}
\def\arraystretch{1.2}
\begin{longtable}{lllp{0.7mm}l}
\caption[caption]{The Lie algebra $\lie{s}_\RR$ from Theorem~\ref{reduction_to_so_n_1} for all simple real Lie algebras~$\lie{g}_\RR$ (continued).
\\\hspace{\textwidth} 
For exceptional $\lie{g}_\RR$, all roots are given by their coordinates in the basis formed by simple roots.} \\
    $\lie{g}$ & $\lie{g}_\RR$ & \multicolumn{2}{l}{\parbox[b]{2cm}{Isomorph. type of~$\lie{s}_\RR$}} & \parbox[b]{4.5cm}{Roots of each summand of $\lie{s}$ \\ (in $(\alpha_1, \ldots, \alpha_r)$ basis).} \\ \midrule
\endhead
    \topic{$E_6$}
               & \multirow[t]{5}{*}{$\underset{\text{(split)}}{E I}$ or $\underset{\text{(quasi-split)}}{E II}$}
                 & $\lie{so}(1,2)$ & & \colorbox{yellow}{$(0,0,0,1,0,0)$} \\*
    \lasttopic & & $\oplus\; \lie{so}(1,2)$ & & \colorbox{yellow}{$(0,0,1,1,1,0)$} \\*
    \lasttopic & & $\oplus\; \lie{so}(1,2)$ & & \colorbox{yellow}{$(1,0,1,1,1,1)$} \\*
    \lasttopic & & $\oplus\; \lie{so}(1,2)$ & & \colorbox{yellow}{$(1,2,2,3,2,1)$} \\*
    \lasttopic & & $\oplus\; \RR^2$ & & \\ \cmidrule{2-5}
    \lasttopic & \multirow[t]{6}{*}{$E III$}
                 & $\lie{so}(1,2)$ & & \colorbox{yellow}{$(1,0,1,1,1,1)$} \\*
    \lasttopic & & $\oplus\; \lie{so}(1,2)$ & & \colorbox{yellow}{$(1,2,2,3,2,1)$} \\*
    \lasttopic & & \multirow{3}{*}{$\oplus\; \lie{su}(4)$}
                 & \ldelim\{{3}{*} & \colorbox{white}{$(0,1,0,0,0,0)$,} \\*
    \lasttopic & & & &                 \colorbox{white}{$(0,0,1,0,0,0)$,} \\*
    \lasttopic & & & &                 \colorbox{white}{$(0,0,0,1,0,0)$} \\*
    \lasttopic & & $\oplus\; \RR$ & \\ \cmidrule{2-5}
    \lasttopic & \multirow[t]{6}{*}{$E IV$}
                 & \multirow{5}{*}{$\lie{so}(1,9)$}
                   & \ldelim\{{5}{*} & \colorbox{yellow}{$(1,2,2,3,2,1)$}, \\*
    \lasttopic & & & &                 \colorbox{white}{$(0,1,0,0,0,0)$,} \\*
    \lasttopic & & & &                 \colorbox{white}{$(0,0,0,1,0,0)$,} \\*
    \lasttopic & & & &                 \colorbox{white}{$(0,0,1,0,0,0)$,} \\*
    \lasttopic & & & &                 \colorbox{white}{$(0,0,0,0,1,0)$} \\*
    \lasttopic & & $\oplus\; \RR$ & \\ \midrule
    \topic{$E_7$}
               & \multirow[t]{7}{*}{$\underset{\text{(split)}}{E V}$}
                 & $\lie{so}(1,2)$ & & \colorbox{yellow}{$(0,1,0,0,0,0,0)$} \\*
    \lasttopic & & $\oplus\; \lie{so}(1,2)$ & & \colorbox{yellow}{$(0,0,1,0,0,0,0)$} \\*
    \lasttopic & & $\oplus\; \lie{so}(1,2)$ & & \colorbox{yellow}{$(0,1,1,2,1,0,0)$} \\*
    \lasttopic & & $\oplus\; \lie{so}(1,2)$ & & \colorbox{yellow}{$(0,0,0,0,1,0,0)$} \\*
    \lasttopic & & $\oplus\; \lie{so}(1,2)$ & & \colorbox{yellow}{$(0,1,1,2,2,2,1)$} \\*
    \lasttopic & & $\oplus\; \lie{so}(1,2)$ & & \colorbox{yellow}{$(0,0,0,0,0,0,1)$} \\*
    \lasttopic & & $\oplus\; \lie{so}(1,2)$ & & \colorbox{yellow}{$(2,2,3,4,3,2,1)$} \\ \cmidrule{2-5}
    \lasttopic & \multirow[t]{7}{*}{$E VI$}
                 & $\lie{so}(1,2)$ & & \colorbox{yellow}{$(0,0,1,0,0,0,0)$} \\*
    \lasttopic & & $\oplus\; \lie{so}(1,2)$ & & \colorbox{yellow}{$(0,1,1,2,1,0,0)$} \\*
    \lasttopic & & $\oplus\; \lie{so}(1,2)$ & & \colorbox{yellow}{$(0,1,1,2,2,2,1)$} \\*
    \lasttopic & & $\oplus\; \lie{so}(1,2)$ & & \colorbox{yellow}{$(2,2,3,4,3,2,1)$} \\*
    \lasttopic & & $\oplus\; \lie{so}_3(\RR)$ & & \colorbox{white}{$(0,1,0,0,0,0,0)$} \\*
    \lasttopic & & $\oplus\; \lie{so}_3(\RR)$ & & \colorbox{white}{$(0,0,0,0,1,0,0)$} \\*
    \lasttopic & & $\oplus\; \lie{so}_3(\RR)$ & & \colorbox{white}{$(0,0,0,0,0,0,1)$} \\ \cmidrule{2-5}
    \lasttopic & \multirow[t]{7}{*}{$E VII$}
                 & $\lie{so}(1,2)$ & & \colorbox{yellow}{$(0,0,0,0,0,0,1)$} \\*
    \lasttopic & & $\oplus\; \lie{so}(1,2)$ & & \colorbox{yellow}{$(0,1,1,2,2,2,1)$} \\*
    \lasttopic & & $\oplus\; \lie{so}(1,2)$ & & \colorbox{yellow}{$(2,2,3,4,3,2,1)$} \\*
    \lasttopic & & \multirow{4}{*}{$\oplus\; \lie{so}_8(\RR)$}
                   & \ldelim\{{4}{*} & \colorbox{white}{$(0,0,1,0,0,0,0)$,} \\*
    \lasttopic & & & &                 \colorbox{white}{$(0,0,0,1,0,0,0)$,} \\*
    \lasttopic & & & &                 \colorbox{white}{$(0,1,0,0,0,0,0)$,} \\*
    \lasttopic & & & &                 \colorbox{white}{$(0,0,0,0,1,0,0)$} \\ \midrule
    \topic{$E_8$}
               & \multirow[t]{8}{*}{$\underset{\text{(split)}}{E VIII}$}
                 & $\lie{so}(1,2)$ & & \colorbox{yellow}{$(0,1,0,0,0,0,0,0)$} \\*
    \lasttopic & & $\oplus\; \lie{so}(1,2)$ & & \colorbox{yellow}{$(0,0,1,0,0,0,0,0)$} \\*
    \lasttopic & & $\oplus\; \lie{so}(1,2)$ & & \colorbox{yellow}{$(0,1,1,2,1,0,0,0)$} \\*
    \lasttopic & & $\oplus\; \lie{so}(1,2)$ & & \colorbox{yellow}{$(0,0,0,0,1,0,0,0)$} \\*
    \lasttopic & & $\oplus\; \lie{so}(1,2)$ & & \colorbox{yellow}{$(0,1,1,2,2,2,1,0)$} \\*
    \lasttopic & & $\oplus\; \lie{so}(1,2)$ & & \colorbox{yellow}{$(0,0,0,0,0,0,1,0)$} \\*
    \lasttopic & & $\oplus\; \lie{so}(1,2)$ & & \colorbox{yellow}{$(2,3,4,6,5,4,3,2)$} \\*
    \lasttopic & & $\oplus\; \lie{so}(1,2)$ & & \colorbox{yellow}{$(2,2,3,4,3,2,1,0)$} \\ \cmidrule{2-5}
    \lasttopic & \multirow[t]{8}{*}{$E IX$}
                 & $\lie{so}(1,2)$ & & \colorbox{yellow}{$(0,0,0,0,0,0,1,0)$} \\*
    \lasttopic & & $\oplus\; \lie{so}(1,2)$ & & \colorbox{yellow}{$(0,1,1,2,2,2,1,0)$} \\*
    \lasttopic & & $\oplus\; \lie{so}(1,2)$ & & \colorbox{yellow}{$(2,2,3,4,3,2,1,0)$} \\*
    \lasttopic & & $\oplus\; \lie{so}(1,2)$ & & \colorbox{yellow}{$(2,3,4,6,5,4,3,2)$} \\*
    \lasttopic & & \multirow{4}{*}{$\oplus\; \lie{so}_8(\RR)$}
                   & \ldelim\{{4}{*} & \colorbox{white}{$(0,0,1,0,0,0,0,0)$,} \\*
    \lasttopic & & & &                 \colorbox{white}{$(0,0,0,1,0,0,0,0)$,} \\*
    \lasttopic & & & &                 \colorbox{white}{$(0,1,0,0,0,0,0,0)$,} \\*
    \lasttopic & & & &                 \colorbox{white}{$(0,0,0,0,1,0,0,0)$} \\ \midrule
    \topic{$F_4$}
               & \multirow[t]{4}{*}{$\underset{\text{(split)}}{F I}$}
                 & $\lie{so}(1,2)$ & & \colorbox{yellow}{$(0,1,0,0)$} \\*
    \lasttopic & & $\oplus\; \lie{so}(1,2)$ & & \colorbox{yellow}{$(0,1,2,0)$} \\*
    \lasttopic & & $\oplus\; \lie{so}(1,2)$ & & \colorbox{yellow}{$(0,1,2,2)$} \\*
    \lasttopic & & $\oplus\; \lie{so}(1,2)$ & & \colorbox{yellow}{$(2,3,4,2)$} \\ \cmidrule{2-5}
    \lasttopic & \multirow[t]{4}{*}{$F II$}
                 & \multirow{4}{*}{$\lie{so}(1,8)$}
                   & \ldelim\{{4}{*} & \colorbox{yellow}{$(0,1,2,2)$}, \\*
    \lasttopic & & & &                 \colorbox{white}{$(1,0,0,0)$,} \\*
    \lasttopic & & & &                 \colorbox{white}{$(0,1,0,0)$,} \\*
    \lasttopic & & & &                 \colorbox{white}{$(0,0,1,0)$} \\ \midrule
    \topic{$G_2$}
               & \multirow[t]{2}{*}{$\underset{\text{(split)}}{G}$}
                 & $\lie{so}(1,2)$ & & \colorbox{yellow}{$(1,0)$} \\*
    \lasttopic & & $\oplus\; \lie{so}(1,2)$ & & \colorbox{yellow}{$(3,2)$} \\
    \bottomrule
\end{longtable}
\end{small}

\subsubsection*{Case where $\lie{g}_\RR$ is absolutely simple}

For each absolutely simple real Lie algebra $\lie{g}_\RR$, an expression of the projection map from $\lie{h}$ to $\lie{a}$ in coordinates can be found for example in the column entitled ``$r$'' of \cite[Table~9]{OV90} (beware however that the authors' root numbering \cite[Table~1]{OV90} strongly differs from Bourbaki's). From there, one can then easily find the full (non-restricted) root system of $\lie{s}_\RR$, and then extract its simple roots. The results are tabulated in Table~\ref{tab:s_roots}.

\begin{remark}
For the classical algebras $\lie{g}_\RR$, the description of $\lie{s}_\RR$ can be simplified. Using exceptional isomorphisms, we can see that $\lie{s}_\RR$ is then as given in Table~\ref{tab:s_roots_rephrased}. The remarkable fact about this expression of~$\lie{s}_\RR$ is that the embedding $\lie{s}_\RR \to \lie{g}_\RR$ is then (up to conjugation) the ``obvious'' diagonal map.

In other terms, $\lie{s}_\RR$ is equal (as a subalgebra of $\lie{g}_\RR$) to the stabilizer of some decomposition $V_{\operatorname{def}} = V_1 \oplus \cdots \oplus V_x$ of the defining representation $V_{\operatorname{def}}$ of $\lie{g}_\RR$ (as given e.g.\ in \cite[Table~8]{OV90}). Here each subspace $V_i$ is the defining representation of the $i$-th direct summand of $\lie{s}_\RR$. Note however that in type $A_r$, $\lie{s}_\RR$ has an additional abelian summand, isomorphic to $\RR^{x-1}$: it comprises the (traceless) block-diagonal elements that act by homothety on each of the subspaces $V_1, \ldots, V_x$.
\end{remark}

\begin{table}[p]
  \caption[caption]{\label{tab:s_roots_rephrased} The Lie algebra $\lie{s}_\RR$ from Theorem~\ref{reduction_to_so_n_1} for all classical simple real Lie algebras~$\lie{g}_\RR$, expressed so that $\lie{s}_\RR \to \lie{g}_\RR$ is conjugate to the diagonal embedding.}
  \centering\bigskip
  \makebox[\textwidth][c]{
  \begin{tabular}[t]{lll}
    $\lie{g}$ & $\lie{g}_\RR$ & $\lie{s}$ \\
    \midrule
    \multirow[t]{5}{*}{$\underset{r \geq 1}{A_r}$}
    & $\lie{sl}_{r+1}(\RR)$ & $\lie{sl}_2(\RR)^{\lfloor \frac{r+1}{2} \rfloor} \;\oplus\; \RR^{\lfloor \frac{r}{2} \rfloor}$ \\ \cmidrule{2-3}
    & \makecell[l]{$\lie{su}(p,r+1-p)$\\ \quad \footnotesize $0 \leq p < \frac{r+1}{2}$}
      & $\lie{su}(1,1)^p \;\oplus\; \lie{su}(r+1-2p) \;\oplus\; \RR^p$ \\ \cmidrule{2-3}
    & \makecell[l]{$\lie{su}(p,p)$\\ \quad \footnotesize $2p - 1 = r$}
      & $\lie{su}(1,1)^p \;\oplus\; \RR^{p-1}$ \\ \cmidrule{2-3}
    & \makecell[l]{$\lie{sl}_m(\HH)$\\ \quad \footnotesize $2m - 1 = r$, $m$ even}
      & $\lie{sl}_2(\HH)^{\frac{m}{2}} \;\oplus\; \RR^{\frac{m}{2}-1}$ \\ \cmidrule{2-3}
    & \makecell[l]{$\lie{sl}_m(\HH)$\\ \quad \footnotesize $2m - 1 = r$, $m$ odd}
      & $\lie{sl}_2(\HH)^{\frac{m-1}{2}} \;\oplus\; \lie{sl}_1(\HH) \;\oplus\; \RR^{\frac{m-1}{2}}$ \\
    \midrule
    \multirow[t]{2}{*}{$\underset{n \geq 3,\; n \neq 4}{B_{\frac{n-1}{2}} \text{ or } D_{\frac{n}{2}}}$}
    & \makecell[l]{$\lie{so}(p,n-p)$\\ \quad \footnotesize $0 \leq p \leq \frac{n}{2}$, $p$ even}
      & $\lie{so}(2,2)^{\frac{p}{2}} \;\oplus\; \lie{so}_{n-2p}(\RR)$ \\ \cmidrule{2-3}
    & \makecell[l]{$\lie{so}(p,n-p)$\\ \quad \footnotesize $0 \leq p \leq \frac{n}{2}$, $p$ odd}
      & $\lie{so}(2,2)^{\frac{p-1}{2}} \;\oplus\; \lie{so}(1, n+1-2p)$ \\
    \midrule
    \multirow[t]{2}{*}{$\underset{r \geq 1}{C_r}$}
    & $\lie{sp}_{\twice r}(\RR)$
      & $\lie{sp}_{\twice 1}(\RR)^r$ \\ \cmidrule{2-3}
    & \makecell[l]{$\lie{sp}_\twice(p,r-p)$\\ \quad \footnotesize $0 \leq p \leq \frac{r}{2}$}
      & $\lie{sp}_\twice(1,1)^p \;\oplus\; \lie{sp}_\twice(r-2p)$ \\
    \midrule
    \multirow[t]{2}{*}{$\underset{r \geq 3}{D_r}$}
    & \makecell[l]{$\lie{so}^*(2r)$\\ \quad \footnotesize $r$ even}
      & $\lie{so}^*(4)^{\frac{r}{2}}$ \\ \cmidrule{2-3}
    & \makecell[l]{$\lie{so}^*(2r)$\\ \quad \footnotesize $r$ odd}
      & $\lie{so}^*(4)^{\frac{r-1}{2}} \;\oplus\; \lie{so}^*(2)$ \\
    \bottomrule
  \end{tabular}
  }
\end{table}

\subsection*{Case where $\lie{g}_\RR$ is complex}

Suppose now that the real Lie algebra $\lie{g}_\RR$ is in fact complex, or, to be more precise, is obtained from some complex Lie algebra $\lie{g}_\CC$ by restriction of scalars. This case is in theory straightforward, so was not included in the table. Essentially, the result is then the same as it would be if we replaced $\lie{g}_\RR$ by two copies of the split real form of $\lie{g}_\CC$.

In practice, however, restriction of scalars has tremendous potential for confusion; so let us still spell out some details.

We have a canonical isomorphism
\[\lie{g}_\RR^\CC \simeq \lie{g}_\CC \oplus \overline{\lie{g}_\CC},\]
where $\overline{\lie{g}_\CC}$ coincides with $\lie{g}_\CC$ as a set and as a real vector space, has the same Lie bracket, but has the opposite complex structure (i.e.\ the multiplication-by-$i$ map is negated). Under this isomorphism, the subalgebra $\lie{g}_\RR \subset \lie{g}_\RR^\CC$ identifies with the image of $\lie{g}_\CC$ under the diagonal embedding
\[\fundef{\delta:}{\lie{g}_\CC}{\lie{g}_\CC \oplus \overline{\lie{g}_\CC}}{X}{(X, X)}.\]
(This embedding is thus $\CC$-linear on the first coordinate and $\CC$-antilinear on the second coordinate --- although of course, when talking about $\lie{g}_\RR$ we forget about the complex structure).

Choose some Cartan subspace $\lie{h}_\CC$ of $\lie{g}_\CC$. This then leads to a natural choice of a Cartan subalgebra and subspace for $\lie{g}_\RR$ and $\lie{g}_\RR^\CC$:
\begin{itemize}
\item We choose in $\lie{g}_\RR^\CC \simeq \lie{g}_\CC \oplus \overline{\lie{g}_\CC}$ the Cartan subalgebra $\lie{h} = \lie{h}_\CC \oplus \overline{\lie{h}_\CC}$.
\item We can set the real Cartan subspace $\lie{a}_\RR$ of $\lie{g}_\RR$ to be the $\RR$-linear span of the roots of $\lie{g}_\CC$ over $\lie{h}_\CC$. (This can also be seen as the Cartan subspace of some split real form of $\lie{g}_\CC$).
\item Its complexification $\lie{a}$ is then the image of $\lie{h}_\CC$ by the embedding $\delta: \lie{g}_\CC \to \lie{g}_\CC \oplus \overline{\lie{g}_\CC}$. As for $\lie{h}_{(\RR)}$, it identifies with $\lie{a}_\RR \oplus \lie{a}_\RR \subset \lie{g}_\RR \oplus \lie{g}_\RR \subset \lie{g}_\CC \oplus \overline{\lie{g}_\CC}$.
\end{itemize}

We now recall the classification of the irreducible complex representations of $\lie{g}_\RR$. (Caution: these are not to be confused with the complex representations of $\lie{g}_\CC$! Both correspond to Lie algebra morphisms from $\lie{g}_\CC$ to $\lie{gl}_n(\CC)$ for some $n$, but the latter must be $\CC$-linear whereas the former are only required to be $\RR$-linear, and so are much more numerous.)
\begin{proposition}~
\begin{hypothenum}
\item \cite[Prop.~7.15]{Kna96first} The irreducible complex representations of $\lie{g}_\RR$ are in bijection (via restriction) with the irreducible complex representations of $\lie{g}^\CC_\RR \simeq \lie{g}_\CC \oplus \overline{\lie{g}_\CC}$;
\item {}[classical] The latter are all of the form $V_{\lambda_1}(\lie{g}_\CC) \otimes V_{\lambda_2}(\overline{\lie{g}_\CC})$, where $(\lambda_1, \lambda_2)$ can be any pair of integral dominant weights of $\lie{g}_\CC$.
\end{hypothenum} 
\end{proposition}

Now it is easy to see that with these choices, $\lie{l}(\lie{g}_\RR) = \lie{h}$ and $w_0(\lie{g}_\RR) = (w_0(\lie{g}_\CC), w_0(\lie{g}_\CC))$. It follows that, for every integral dominant weight $(\lambda_1, \lambda_2)$ of $\lie{g}_\RR^\CC$ (where $\lambda_1$ and $\lambda_2$ are integral dominant weights of $\lie{g}_\CC$), we have
\[\restr{w_0}{V^{\lie{l}}_{(\lambda_1, \lambda_2)}(\lie{g}_\RR^\CC)} = \restr{w_0}{V^{\lie{l}}_{\lambda_1}(\lie{g}_\CC)} \otimes \restr{w_0}{V^{\lie{l}}_{\lambda_2}(\overline{\lie{g}_\CC})}.\]

Now for the complex Lie algebra $\lie{g}_\CC$ (with its complex structure), the action of $w_0$ on $V^{\lie{l}}$ is obviously the same as for its split real form.

\begin{remark}
Of course this also applies to the algebra $\lie{so}(1,3) \simeq \lie{sl}_2(\CC)$ studied in Theorem~\ref{algo_for_so_n_1}. The reader can verify that the procedure outlined here provides an alternative proof of Theorem~\ref{algo_for_so_n_1} for $n = 3$.
\end{remark}

\subsection*{A remark about split and quasi-split real forms}

Note that all non absolutely simple Lie algebras are in particular quasi-split. In fact, the previous discussion is a special case of the following more general statement:
\begin{claim}
Let $\lie{g}_{\operatorname{split}}$ and $\lie{g}_{\operatorname{quasi-split}}$ be two real forms of the same complex Lie algebra~$\lie{g}$, which are respectively split and quasi-split. Then (by definition) both of them have the same $\lie{l}$, namely $\lie{l} = \lie{h}$. Furthermore, they also have the same element $w_0$. So the action of $w_0$ on $V^{\lie{l}}$ is in fact the same in the split and quasi-split cases.
\end{claim}
This can be checked at least on a case-by-case basis. (By thinking a bit harder, it may also possible to find an abstract argument for this. \cite[Prop.~5.2]{Heck} seems to be a good starting point.)

\bibliographystyle{alpha}
\bibliography{/home/ilia/Documents/Travaux_mathematiques/mybibliography}

\noindent I. Smilga, Mathematical Institute, University of Oxford, United Kingdom 

\noindent E-mail: \url{ilia.smilga@normalesup.org}
\end{document}